\documentclass[11pt]{article}
\usepackage[english]{babel}
\usepackage{amssymb,amsmath,amsthm}
\usepackage{graphicx}
\newtheorem{theorem}{Theorem}[section]
\newtheorem{lemma}[theorem]{Lemma}
\newtheorem{proposition}[theorem]{Proposition}
\newtheorem{corollary}[theorem]{Corollary}

\newtheorem{example}[theorem]{Example}
\newtheorem{question}[theorem]{Question}
\newtheorem{conjecture}[theorem]{Conjecture}

{\theoremstyle{definition}}
{\theoremstyle{definition}}
{\theoremstyle{definition}}

\numberwithin{equation}{section}

\def\Z{{\mathbb Z}}
\def\R{{\mathbb R}}

\def\K{{\mathbb K}}

\def\epsilon{\varepsilon}
\def\kappa{\varkappa}
\def\phi{\varphi}
\def\leq{\leqslant}
\def\geq{\geqslant}

\def\dim{{\rm dim}\,}

\def\ker{\hbox{\tt ker}\,}
\def\spann{\hbox{\rm span}\,}

\def\huhu{{\vrule height10pt depth6pt width0pt}}
\def\Huhu{{\vrule height14pt depth6pt width0pt}}
\def\huHu{{\vrule height0pt depth10pt width0pt}}

\title{Two problems from the Polishchuk and Positselski book on Quadratic algebras}

\author{Natalia Iyudu and Stanislav Shkarin}

\begin{document}

\maketitle

\begin{abstract}
In the book 'Quadratic algebras' by Polishchuk and Positselski \cite{popo} algebras  with a small number of generators $(n=2,3)$ are considered. For some number $r$ of relations possible Hilbert series  are listed, and those appearing as series of Koszul algebras are specified. The first case, where it was not possible to do, namely the
case of three generators  $n=3$ and six relations $r=6$ is formulated as an open problem. We give here a complete answer to this question, namely for quadratic algebras with $\dim A_1=\dim A_2=3$, we list all possible Hilbert series, and find out which of them can come from Koszul algebras, and which can not.

As a consequence of this classification, we found an algebra, which serves as a counterexample to another problem from the same book \cite{popo} (Chapter 7, Sec. 1, Conjecture 2), saying that Koszul algebra  of finite global homological dimension $d$ has $\dim A_1 \geq d$.
Namely, the 3-generated algebra $A$ given by relations $xx+yx=xz=zy=0$ is Koszul and its Koszul dual algebra $A^!$ has Hilbert series of degree 4: $H_{A^!}(t)=  1+3t+3t^2+2t^3+t^4$, hence $A$ has global homological dimension 4.

\end{abstract}



\section{Introduction \label{s1}}\rm

Quadratic algebras have been studied intensely during the past several decades. Being interesting in their own right, they have many important  applications in various parts of mathematics and physics including algebraic geometry, algebraic topology and group theory as well as in mathematical physics.

They frequently originate in physics. One example is that the 3-dimensional Sklyanin algebras were introduced and used in order to integrate a wide class of quantum systems on a lattice. These algebras have their Koszul duals  in the class of quadratic algebras with $\dim A_1=\dim A_2=3$, the very class we study in this paper.

Quadratic algebras are noncommutative objects which lie in foundation of many noncommutative theories, for example, in work of A.Connes and M.Dubois-Violette \cite{CDV}, notions of noncommutative differential geometry obtain their purely algebraic counterpart through introducing the appropriate quadratic form on quadratic algebras.
The big area of research generalizing notions of algebraic geometry to noncommutative spaces, due to Artin, Tate, Van den Bergh, Stafford etc. \cite{AS, ATV1,ATV2,sue} contains a great deal of studying structural and homological properties of quadratic algebras and  their representations.
 Certain quadratic algebras serve as important examples for the notions of noncommutative (symplectic) spaces introduced by Kontsevich \cite{MK,MR}, so information about general rules on the structure of such algebras makes it possible to describe examples explicitly.

 We find it very important to study fundamental, most general properties of quadratic algebras, their Hilbert series, Koszulity, other homological properties, PBW type properties, etc. and to develop appropriate tools for that, which is a goal of present paper.

For further information on quadratic algebras, their Hilbert series and various aspects of their applications, we refer to \cite{ani1,ani2,cana,na,ns,ns5,ns1,popo,
dr,M,Pri,gosh,skl,odf,DV1,DV2,BW,Ag1,Zelm,Er} and references therein, however it will give still the list which is far from being exhaustive.

Throughout this paper, $\K$ is an arbitrary field. For a $\Z_+$-graded vector space $B$, $B_m$ always stands for the $m^{\rm th}$ component of $B$ and $\textstyle H_B(t)=\sum\limits_{j=0}^\infty \dim B_j\,\,t^j$ is  the {\it Hilbert series} of $B$.

If $V$ is an $n$-dimensional vector space over $\K$, then $F=F(V)$ is the tensor algebra of $V$, which is naturally identified with the free algebra $\K\langle x_1,\dots,x_n\rangle$ for any choice of a basis $x_1,\dots,x_n$ in $V$.
We often use the juxtaposition notation for the operation in $F(V)$ (for instance, we write $x_jx_k$ instead of $x_j\otimes x_k$). We always use the degree grading on $F$: the $m^{\rm th}$ graded component $F_m$ of $F$ is $V^m$. A {\it degree graded} algebra $A$ is a quotient of $F$ by a proper graded ideal $I$ ($I$ is graded if it is the direct sum of $I\cap F_m$). This ideal is called the {\it ideal of relations} of $A$. If $x_1,\dots,x_n$ is a fixed basis in $V$ and the monomials in $x_j$ carry an ordering compatible with the multiplication in $A$, we can speak of the Gr\"obner basis of the ideal of relations of $A$. If $\Lambda$ is the set of leading monomials of the elements of such a basis, then the {\it normal words} for $A$ are the monomials in $x_j$ featuring no element of $\Lambda$ as a submonomial. Normal words form a basis in $A$ as a $\K$-vector space. Thus, knowing the normal words implies knowing the Hilbert series.

If $R$ is a subspace of the $n^2$-dimensional space $V^2$, then the quotient of $F$ by the ideal $I(V,R)$ generated by $R$ is called a {\it quadratic algebra} and denoted $A(V,R)$. Following \cite{popo}, we say that a quadratic algebra $A=A(V,R)$ is a {\it PBW-algebra} if there are linear bases $x_1,\dots,x_n$ and $g_1,\dots,g_m$ in $V$ and $R$ respectively such that $g_1,\dots,g_m$ is a Gr\"obner basis of the ideal $I=I(V,R)$ with respect to some  well-ordering on the monomials in $x_1,\dots,x_n$, compatible with multiplication. For a given basis $x_1,\dots,x_n$ in $V$, we get a bilinear form on $\K\langle x_1,\dots,x_n\rangle$ by setting $[u,v]=\delta_{u,v}$ for every pair of monomials $u$ and $v$ in $x_1,\dots,x_n$. The quadratic algebra $A^!=A(V,R^\perp)$, where $R^\perp=\{u\in V^2:[r,u]=0\ \text{for each}\ r\in R\}$, is called the
{\it Koszul dual algebra} of $A$, we will call it also just {\it dual algebra}. Note that up to an isomorphism, the graded algebra $A^!$ does not depend on the choice of a basis in $V$. It is well-known that $A$ is PBW if and only if $A^!$ is PBW.

A degree graded algebra $A$ is called {\it Koszul} if the graded left $A$-module $\K$ (the structure is provided by the augmentation map) has a free resolution $\dots\to M_m\to\dots\to M_1\to A\to\K\to 0$ with the second last arrow being the augmentation map, and with each $M_k$ generated in degree $k$. The latter means that the matrices of the maps $M_m\to M_{m-1}$ with respect to some free bases consist of homogeneous elements of degree $1$. Replacing left modules by the right ones leads to the same class of algebras. We use the following well-known properties of Koszul algebras:
\begin{align}\notag
&\text{every Koszul algebra is quadratic; every PBW-algebra is Koszul;}
\\
\notag
&\text{$A$ is Koszul $\iff$ $A^!$ is Koszul;}
\\
&\text{$H_A(-t)H_{A^!}(t)=1$ if $A$ is Koszul.}
\label{stm2}
\end{align}

In \cite[Chapter~6, Section~5]{popo} possible Hilbert series of Koszul algebras $A$ with small values of $\dim A_1$ and $\dim A_2$ are listed. The first case not covered there is $\dim A_1=\dim A_2=3$. In this case, only the Hilbert series of PBW algebras are given. It is stated in \cite{popo} that the complete list of Hilbert series of quadratic algebras satisfying $\dim A_1=\dim A_2=3$ as well as the complete list of the Hilbert series of Koszul algebras in this case are unknown. We fill this gap by proving the following results.

\begin{theorem}\label{main001}
For quadratic algebras $A$ satisfying $\dim A_1=\dim A_2=3$, the complete list of possible Hilbert series is $\{H_1,\dots,H_{11}\}$, where\hfill\break
\smallskip
\centerline{$\begin{array}{ll}
\text{$H_1(t){=}1{+}3t{+}3t^2;$}\Huhu&
\text{$H_2(t){=}1{+}3t{+}3t^2{+}t^3=(1+t)^3;$}\\
\text{$H_3(t){=}1{+}3t{+}3t^2{+}t^3{+}t^4{+}t^5{+}\cdots{=}\frac{1{+}2t{-}2t^3}{1-t};$}&
\text{$H_4(t){=}1{+}3t{+}3t^2{+}2t^3;$}\\
\text{$H_5(t){=}1{+}3t{+}3t^2{+}2t^3{+}t^4;$}&
\text{$H_6(t){=}1{+}3t{+}3t^2{+}2t^3{+}t^4{+}t^5{+}\cdots{=}\frac{1{+}2t{-}t^3{-}t^4}{1-t};$}\\
\text{$H_7(t){=}1{+}3t{+}3t^2{+}2t^3{+}2t^4{+}2t^5{+}\cdots{=}\frac{1{+}2t{-}t^3}{1-t};$}&
\text{$H_8(t){=}1{+}3t{+}3t^2{+}3t^3{+}3t^4{+}3t^5{+}\cdots{=}\frac{1{+}2t}{1-t};$}\\
\text{$H_9(t){=}1{+}3t{+}3t^2{+}4t^3{+}4t^4{+}4t^5{+}\cdots{=}\frac{1{+}2t{+}t^3}{1-t};$}\ &
\text{$\!\!H_{10}(t){=}1{+}3t{+}3t^2{+}4t^3{+}5t^4{+}6t^5{+}\cdots{=}\frac{1{+}t-2t^2{+}t^3}{(1-t)^2};$}\\
\rlap{\text{$\!\!H_{11}(t){=}1{+}3t{+}3t^2{+}5t^3{+}8t^4{+}13t^5{+}\cdots{=}\frac{1{+}2t{-}t^2{-}t^3}{1-t-t^2},$}}&\huHu
\end{array}$}
\noindent where the last series, starting from the third term, is formed by consecutive Fibonacci numbers.
\end{theorem}

As we have mentioned above, the complete list of Hilbert series of PBW algebras $A$ satisfying $\dim A_1=\dim A_2=3$ can be found in \cite{popo}. It consists of $H_j$ with $j\in\{2,7,8,9,10,11\}$.

\begin{theorem}\label{main002}
For Koszul algebras $A$ satisfying $\dim A_1=\dim A_2=3$, the complete list of Hilbert series consists of $H_j$ with $j\in\{2,5,6,7,8,9,10,11\}$. That is, for each $j\in \{2,5,6,7,8,9,10,11\}$, there is a Koszul algebra $A$ satisfying $H_A=H_j$, while for $j\in\{1,3,4\}$, every quadratic algebra $A$ satisfying $H_A=H_j$ is non-Koszul.
\end{theorem}

As a consequence of these classification Theorems \ref{main001} and \ref{main002}, we can find an algebra, which serves as a counterexample to another problem formulated in the Polishchuk and Positselski book \cite{popo}.

\begin{conjecture}\label{conj} {\rm (\cite{popo}, Chapter 7, Sec. 1, Conjecture 2)}
Any Koszul algebra  of finite global homological dimension $d$ has $\dim A_1 \geq d$.   By duality this is equivalent to the following statement: for a Koszul algebra $B$ with $B_{d+1}=0$ and $B_{d} \neq 0$ one has $\dim B_1 \geq d$.
\end{conjecture}

We can consider algebra $A_5$ from the Table 1. It is given by relations $xx-yx=xy=yy=yx=zx=zz=0$. The Hilbert series of this algebra is $H_A(t)= H_5= 1+3t+3t^2+2t^3+t^4$, so it is a polynomial of degree 4. Then quite straightforward arguments in Proposition~\ref{kokoko} ensure that the Koszul dual algebra $A_5^!$, given by relations $xx+yx=xz=zy=0$, is Koszul. This provides a counterexample to  Conjecture~\ref{conj}. Namely, the algebra $A^!$ serves as a counterexample to the first part of the conjecture: it has a global homological dimension 4. The algebra $A$ itself is a counterexample to the second statement of the conjecture, since its Hilbert series is a polynomial of degree 4.

Note that the list of series acquires two extra members $H_5$ and $H_6$ when the PBW condition is relaxed to Koszulity.

The key lemma, allowing to manage all possibilities, is the following linear algebra statement.

\begin{lemma}\label{split} Let $V$ be a $3$-dimensional vector space over an infinite field $\K$ and $R$ be a $6$-dimensional subspace of $V\otimes V$. Then at least one of the following statements is true$:$
\begin{itemize}\itemsep=-3pt
\item[{\rm (P1)}] there is a $1$-dimensional subspace $L\subset V$ such that $(V\otimes L)\oplus R=V\otimes V$ or $(L\otimes V)\oplus R=V\otimes V;$
\item[{\rm (P2)}] there is $1$-dimensional subspace $L\subset V$ such that $V\otimes L\subset R$ or $L\otimes V\subset R;$
\item[{\rm (P3)}] there is an invertible linear map $T:V\to V$ such that $R=\spann\{x\otimes Tx:x\in V\}$.
\end{itemize}
\end{lemma}

While (P1) and (P2) are not mutually exclusive, (P3) is incompatible with each of (P1) and (P2).

 In Section~2 we show that the Hilbert series of any quadratic algebra $A$, satisfying $\dim A_1=\dim A_2=3$ belongs to $\{H_1,\dots,H_{11}\}$, applying Lemma~\ref{split} and Gr\"obner basis techniques. Then in Section~3 we give a proof of Lemma~\ref{split}.
 In Section~4 we show that $H_j$ is the Hilbert series of a quadratic algebra $A$ for $1\leq j\leq11$. We also observe that $A$ can be chosen Koszul if $j\in\{2,5,6,7,8,9,10,11\}$ and that every algebra $A$ with $H_A=H_j$ for $j\in\{1,3,4\}$ is non-Koszul, thus completing the proofs of Theorems~\ref{main001} and~\ref{main002}.

Throughout the paper, when talking of Gr\"obner bases,
assume that the monomials carry the {\bf left-to-right degree lexicographical ordering} with the variables ordered by $x>y>z$ or $x_1>x_2>x_3$ (depending on how the variables are called in each case).

\section{Admissible series}

In this section we apply Lemma~\ref{split} and Gr\"obner basis arguments  to prove the following result. Next section will be dedicated to the proof of the Lemma~\ref{split} itself.

\begin{proposition}\label{adse} Let $A$ be a quadratic $\K$-algebra satisfying $\dim A_1=\dim A_2=3$. Then $H_A\in \{H_1,\dots,H_{11}\}$.
\end{proposition}

Since replacing the ground field $\K$ by a field extension does not change the Hilbert series of an algebra given by generators and relations, for the purpose of proving Proposition~\ref{adse}, we can without loss of generality assume that $\K$ is algebraically closed. Then $\K$ is infinite. By Lemma~\ref{split}, Proposition~\ref{adse} is an immediate corollary of the  following three lemmas. We essentially consider three possibilities given by Lemma~\ref{split}, and in each case find out which series are possible, looking mainly at the shape of the Gr\"obner basis.

\begin{lemma}\label{I} Let $V$ be a $3$-dimensional vector space over $\K$ and $R$ be a $6$-dimensional subspace of $V\otimes V$ such that condition {\rm (P1)} of Lemma~$\ref{split}$ is satisfied. Then for the quadratic algebra $A=A(V,R)$, $H_A\in\{H_1,\dots,H_8\}$.
\end{lemma}

\begin{lemma}\label{II} Let $V$ be a $3$-dimensional vector space over an algebraically closed field $\K$ and $R$ be a $6$-dimensional subspace of $V\otimes V$ such that condition {\rm (P2)} of Lemma~$\ref{split}$ is satisfied, while {\rm (P1)} fails. Then for the quadratic algebra $A=A(V,R)$, $H_A\in\{H_1,H_2,H_3,H_7,H_9,H_{10},H_{11}\}$.
\end{lemma}

\begin{lemma}\label{III} Let $V$ be a $3$-dimensional vector space over $\K$ and $R\subset V\otimes V$ be a subspace satisfying condition {\rm (P3)} of Lemma~$\ref{split}$. Then for the quadratic algebra $A=A(V,R)$, $H_A=H_2$.
\end{lemma}

\begin{proof}[Proof of Lemma~$\ref{I}$] Since (P1) is satisfied, there is a $1$-dimensional subspace $L\subset V$ such that $(V\otimes L)\oplus R=V\otimes V$ or $(L\otimes V)\oplus R=V\otimes V$. These two cases reduce to each other by passing to the algebra with the opposite multiplication. Thus we can assume that $(L\otimes V)\oplus R=V\otimes V$. Pick a basis $x_1,x_2,x_3$ in $V$ such that $x_3$ spans $L$. Since $(L\otimes V)\oplus R=V\otimes V$, there is a linear basis in $R$ of the form (we skip the symbol $\otimes$ for the rest of the proof):
\begin{equation}\label{baR}
r_{j,k}=x_jx_k-x_3u_{j,k}\ \ \text{for $1\leq j\leq 2$ and $1\leq k\leq 3$, where $u_{j,k}\in V$.}
\end{equation}
It follows that in the algebra $A$, $A_2=x_3V=x_3A_1$. Then $A_3=A_2V=x_3VV=x_3A_2=x_3^2V$. Iterating, we get $A_n=x_3A_{n-1}=x_3^{n-1}V$ for each $n\geq 2$. In particular, $\dim A_n\leq \dim A_{n-1}\leq 3$ for each $n\geq 2$.
We also know that $\dim A_1=\dim A_2=3$.

{\bf Case 1:} \ $\dim A_3=3$. This can only happen if $r_{j,k}$ form a Gr\"obner basis of the ideal $I=I(V,R)$. Since the leading monomials of these relations are $x_jx_k$ for $1\leq j\leq 2$ and $1\leq k\leq 3$, the normal words of degree $n\geq 3$ are $x_3^{n-1}x_j$ for $1\leq j\leq 3$. Hence $\dim A_n=3$ for $n\geq 3$ and $H_A=H_8$.

{\bf Case 2:} \ $\dim A_3=2$. This happens when there is exactly one degree $3$ element $g$ of the Gr\"obner basis of $I$.
The leading monomial of $g$ must have the shape $x_3^2x_j$ with $1\leq j\leq 3$. If $j=3$, we have $g=x_3^3$ and $x_3^3=0$ in $A$. Hence for $n\geq 4$, $A_n=x_3^{n-1}V=\{0\}$. Thus $H_A=1+3t+3t^2+2t^3=H_4$. It remains to consider the case $j\in\{1,2\}$. Swapping $x_1$ and $x_2$, if necessary, we can without loss of generality assume that $j=1$. We know that $\dim A_4\leq \dim A_3=2$. The case $\dim A_4=2$ can only happen if the relations $r_{j,k}$ together with $g$ form a Gr\"obner basis of $I$. In this case the normal words of degree $n\geq 3$ are $x_3^{n-1}x_k$ with $k\in\{2,3\}$. This gives $H_A=1+3t+3t^2+2t^3+2t^4+2t^5+{\cdots}=H_7$. It remains to consider the case $\dim A_4=1$. This happens when there is exactly one degree $4$ element $h$ in the Gr\"obner basis of $I$. The leading monomial of $h$ must have the shape $x_3^3x_k$ with $2\leq k\leq 3$. If $k=3$, we have $h=x_3^4$ and $x_3^4=0$ in $A$. Hence for $n\geq 5$, $A_n=x_3^{n-1}V=\{0\}$. Thus $H_A=1+3t+3t^2+2t^3+t^4=H_5$. Assume now that $k=2$. If the relations $r_{j,k}$ together with $g$ and $h$ do not form the Gr\"obner basis of $I$, there is a degree $5$ element $q$ of this Gr\"obner basis. By looking at the leading terms of $r_{j,k}$, $g$ and $h$, we see that the only possibility is for $q$ to be equal $x_3^5$ up to a non-zero scalar multiple. Again, this gives $H_A=H_5$. On the other hand, if $r_{j,k}$ together with $g$ and $h$ do form the Gr\"obner basis of $I$, the only normal word of degree $n\geq 4$ is $x_3^n$. Hence $H_A=1+3t+3t^2+2t^3+t^4+t^5+{\cdots}=H_6$.

{\bf Case 3:} \ $\dim A_3=1$. This happens when there are exactly two degree $3$ elements $g$ and $h$ of the Gr\"obner basis of $I$. By swapping $g$ and $h$, if necessary, we can assume that the leading terms of $g$ and $h$ are $x_3^2x_j$ and $x_3^2x_k$ respectively with $1\leq j<k\leq 3$. If $k=3$, we have $h=x_3^3$ and $x_3^3=0$ in $A$. Hence for $n\geq 4$, $A_n=x_3^{n-1}V=\{0\}$. Thus $H_A=1+3t+3t^2+t^3=H_2$. It remains to consider the case $j=1$, $k=2$. If the relations $r_{j,k}$ together with $g$ and $h$ do not form the Gr\"obner basis of $I$, there is a degree $4$ element $q$ in this Gr\"obner basis. By looking at the leading terms of $r_{j,k}$, $g$ and $h$, we see that the only possibility is for $q$ to be equal $x_3^4$ up to a non-zero scalar multiple. Again, this gives $H_A=H_2$. If $r_{j,k}$ together with $g$ and $h$ do form the Gr\"obner basis of $I$, the only normal word of degree $n\geq 3$ is $x_3^n$. Hence $H_A=1+3t+3t^2+t^3+t^4+t^5+{\cdots}=H_3$.

{\bf Case 4:} \ $\dim A_3=0$. Obviously, $H_A=1+3t+3t^2=H_1$.
\end{proof}

\begin{proof}[Proof of Lemma~$\ref{II}$] Since (P2) holds, there is a $1$-dimensional subspace $L$ of $V$ such that
$VL\subset R$ or $LV\subset R$ (we skip the symbol $\otimes$ throughout the proof). The cases $VL\subset R$ and $LV\subset R$ reduce to each other by passing to the algebra with the opposite multiplication. Thus we can assume that $LV\subset R$. Pick $x$ in $V$, which spans $L$. Since (P1) fails,
\begin{equation}\label{eee1}
\text{for each $u\in V\setminus\{0\}$, there is $v=v(u)\in V\setminus\{0\}$ such that $uv\in R$.}
\end{equation}
We shall verify that there are $y,z\in V$ such that $x,y,z$ is a basis in $V$ and at least one of the following conditions holds:
\begin{align}
R&=\spann\{xx,xy,xz,yx,zx,h\}\ \ \text{with $h\in\{yy,yz-azy,yz-zy+zz\}$ ($a\in\K$)};\label{Q1}
\\
R&=\spann\{xx,xy,xz,yy,zy,h\}\ \ \text{with $h\in\{yx-zz,yz-zx,yx,yz,zx,zz\}$};\label{Q2}
\\
R&=\spann\{xx,xy,xz,yy,zz,h\}\ \ \text{with $h\in\{yx+zx,yz+zy,yx+zx-yz-zy\}$}.\label{Q3}
\end{align}

{\bf Case 1:} \ $VL\subset R$. Pick arbitrary $u,v\in V$ such that $x,u,v$ is a basis in $V$. Then the $5$-dimensional space $LV+VL$ spanned by $S_0=\{xx,xu,xv,ux,vx\}$ is contained in $R$. Since $R$ is $6$-dimensional, it is spanned by $S_0\cup\{f\}$, where $f=auu+buv+cvu+dvv$ with $(a,b,c,d)\in\K^4$, $(a,b,c,d)\neq(0,0,0,0)$. Since $\K$ is algebraically closed, there is a non-zero $(p,s)\in\K^2$ such that $ap^2+(b+c)ps+ds^2=0$. Next, pick $(q,t)\in\K^2$ such that $(p,s)$ and $(q,t)$ are linearly independent. The non-degenerate linear substitution, in which old $u$ and $v$ are replaced by $pu+qv$ and $su+tv$ respectively, transforms $f$ into $g=\alpha uv+\beta vu+\gamma vv$ with non-zero $(\alpha,\beta,\gamma)\in\K^3$. If $\alpha=0$ and $\beta\neq 0$, we set $y=v$ and $z=\beta u+\gamma v$, while if $\alpha\neq 0$ and $\beta=0$, we set $y=\alpha u+\gamma v$ and $z=u$. This substitution transforms $g$ into a (non-zero) scalar multiple of $yz$. Now, with respect to the basis $x,y,z$,  $R$ is spanned by $S\cup\{yz\}$ with $S=\{xx,xy,xz,yx,zx\}$. If $\alpha=\beta=0$, we set $y=v$ and $z=u$ and observe that $R$ is spanned by $S\cup\{yy\}$. If $\alpha\beta\neq 0$ and $\alpha+\beta\neq 0$, we set $y=u+\frac{\gamma v}{\alpha+\beta}$, $z=v$ and observe that $R$ is
spanned by $S\cup\bigl\{yz+\frac{\beta}{\alpha}zy\bigr\}$. If $\alpha\beta\neq 0$ and $\alpha+\beta=\gamma=0$, with respect to $y=u$ and $z=v$, $R$ is spanned by $S\cup\{yz-zy\}$. Finally, if $\alpha\beta\gamma\neq 0$ and $\alpha+\beta=0$, we set $y=\frac{\alpha u}{\gamma}$ and $z=v$, with respect to which $R$ is spanned by $S\cup\{yz-zy+zz\}$. Thus (\ref{Q1}) is satisfied provided $VL\subset R$.

{\bf Case 2:} \ $VM\subset R$ for a $1$-dimensional subspace $M$ of $V$ such that $M\neq L$. Pick $u\in M\setminus\{0\}$. Since $L\neq M$, $x$ and $u$ are linearly independent. For $w\in V$ such that $x,u,w$ is a basis in $V$, $R=\spann(S_0\cup \{f\})$, where $S_0=\{xx,xu,xw,uu,wu\}$ and $f=aux+buw+cwx+dww$ with $(a,b,c,d)\in\K^4$, $(a,b,c,d)\neq(0,0,0,0)$. For $\alpha\in\K^*$ and $p,q\in\K$ we can consider the basis $x,y,z$ in $V$ defined by $u=\alpha y$ and $w=z+px+qy$. A direct computation shows that with respect to this basis, $R=\spann(S\cup \{g\})$, where $S=\{xx,xy,xz,yy,zy\}$ and
$$
g=(a\alpha+bp\alpha+cq+dpq)yx+(b\alpha +dq)yz+(c+dp)zx+dzz.
$$
If $d\neq 0$ and $ad=bc$, by choosing $\alpha=1$, $q=-\frac{b}{d}$ and $p=-\frac{c}{d}$, we turn $g$ into a (non-zero) scalar multiple of $zz$. If $d=0$ and $ad\neq bc$, by choosing $\alpha=\frac{d^2}{bc-ad}$, $q=-\frac{b}{d}$ and $p=-\frac{c}{d}$, we turn $g$ into a scalar multiple of $yx-zz$. If $d=0$ and $bc\neq 0$, by choosing $\alpha=-\frac{c}{b}$, $p=0$ and $q=\frac{a}{b}$, we turn $g$ into a scalar multiple of $yz-zx$. If $b=d=0$ and $c\neq 0$, by choosing $\alpha=1$, $p=0$ and $q=-\frac{a}{c}$, we turn $g$ into a scalar multiple of $zx$. If $c=d=0$ and $b\neq 0$, by choosing $\alpha=1$, $q=0$ and $p=-\frac{a}{b}$, we turn $g$ into a scalar multiple of $yz$. Finally, if $b=c=d=0$, by  choosing $\alpha=1$ and $p=q=0$, we turn $g$ into a scalar multiple of $yx$. Thus (\ref{Q2}) is satisfied provided $VM\subset R$ for a $1$-dimensional subspace $M$ different from $L$.

{\bf Case 3:} \ $VM\not \subseteq R$ for every $1$-dimensional subspace $M$ of $V$. This is precisely the negation of the assumptions of Cases~1 and~2. First, we shall verify that in this case
\begin{equation}\label{CA3}
\text{$yz\notin R$ whenever $x,y,z$ is a basis in $V$.}
\end{equation}
We argue by contradiction. Assume that (\ref{CA3}) fails. Then there are $y,z\in V$ such that $x,y,z$ is a basis in $V$ and  $yz\in R$. By (\ref{eee1}), there are non-zero $(a,b,c),(p,q,r)\in\K^3$ such that $z(ax+by+cz),(y+z)(px+qy+rz)\in R$. The assumption of Case 3 implies linear independence of $z$, $ax+by+cz$ and $px+qy+rz$. Indeed, if they were linearly dependent, using the inclusions $yz,z(ax+by+cz),(y+z)(px+qy+rz)\in R$, one easily finds a non-zero $u\in V$ such that $yu,zu\in R$. Since $xu\in R$, this implies $VM\subset R$ with $M$ being the linear span of $u$. Linear independence of $z$, $ax+by+cz$ and $px+qy+rz$ implies that $aq\neq bp$ and that
\begin{equation}\label{CA31}
R=\spann\{xx,xy,xz,yz,z(ax+by+cz),(y+z)(px+qy+rz)\}.
\end{equation}
Since $\K$ is infinite, we can pick $\theta\in\K\setminus\{0,1\}$. By (\ref{eee1}), there is a non-zero  $(\alpha,\beta,\gamma)\in\K^3$ such that $(y+\theta z)(\alpha x+\beta y+\gamma z)\in R$. By (\ref{CA31}), there exist $c_1,c_2,c_3\in\K$ such that
$$
(y+\theta z)(\alpha x+\beta y+\gamma z)=c_1yz+c_2z(ax+by+cz)+c_3(y+z)(px+qy+rz),
$$
where the equality holds in $\K\langle x,y,z\rangle$. Opening up the brackets in the above display, we obtain
$$
\alpha-pc_3=\beta-qc_3=\gamma-c_1-rc_3=\alpha\theta-ac_2-qc_3=\beta\theta-bc_2-qc_3=\gamma\theta-cc_2-rc_3=0.
$$
Plugging $\alpha=pc_3$ and $\beta=qc_3$ into $\alpha\theta-ac_2-qc_3=\beta\theta-bc_2-qc_3=0$, we get
$bc_2+(1-\theta)qc_3=ac_2+(1-\theta)pc_3=0$. Since $\theta\neq 1$ and $aq\neq bp$, the determinant $(1-\theta)(bp-aq)$ of this system of two linear equations on $c_2$ and $c_3$ is non-zero. Hence $c_2=c_3=0$. Since $\theta\neq 0$, the above display implies $\alpha=pc_3=0$, $\beta=qc_3=0$ and $\gamma=\frac{cc_2+rc_3}{\theta}=0$, which contradicts $(\alpha,\beta,\gamma)\neq (0,0,0)$. This contradiction proves (\ref{CA3}).

Now (\ref{CA3}) together with (\ref{eee1}) imply that
\begin{equation}\label{CA32}
\text{for each $u\in V\setminus L$, there is $(a_u,b_u)\in\K^2\setminus\{(0,0)\}$ such that $u(a_uu+b_ux)\in R$.}
\end{equation}
Observe that $a_u\neq 0$ for a Zarisski generic $u\in V$. Indeed, otherwise $VL\subset R$. Next, $b_u\neq 0$ for a Zarisski generic $u\in V$. Indeed, otherwise $R$ contains the $6$-dimensional space ${\cal S}$ of symmetric elements of $V^2$. Since $R$ also contains $LV$ and $LV\cap {\cal S}$ is one-dimensional, $\dim R\geq 8>6$, which is a contradiction. Thus we can pick $s,t\in V$ such that $x,s,t$ is a basis in $V$ and $a_sb_sa_tb_t\neq 0$. Now, using the inclusions $s(a_ss+b_ss), t(a_tt+b_tx)\in\R$, we can pick $p,q\in\K^*$ such that for $u=ps$ and $v=qt$, $u(x-u),v(x-v)\in R$.
For $a=a_{u+v}$ and $b=b_{u+v}$, according to (\ref{CA32}), we have $(a,b)\neq (0,0)$ and $(u+v)(au+av+bx)\in R$.
Then $R=\spann\{xx,xu,xv,u(x-u),v(x-v),(u+v)(au+av+bx)\}$. Now set $y=x-u$ and $z=x-v$. With respect to the basis $x,y,z$, the last equality can be rewritten as $R=\spann(S\cup\{c(y+z)x-a(yz+zy)\}$, where $S=\{xx,xy,xz,yy,zz\}$ with $c=b+2a$. If $c=0$, then $R=\spann(S\cup\{yz+zy\})$. If $a=0$, then $R=\spann(S\cup\{yx+zx\})$. If $ac\neq 0$, then replacing $y$ and $z$ by $\alpha y$ and $\alpha z$ for an appropriate $\alpha\in\K^*$, we get $R=\spann(S\cup\{yx+zx-yz-zy\})$. Thus (\ref{Q3}) is satisfied provided $VM\not \subseteq R$ for every $1$-dimensional subspace $M$ of $V$.

It remains to determine the Hilbert series of $A=A(V,R)$ when $R$ satisfies one of the conditions (\ref{Q1}), (\ref{Q2}) or (\ref{Q3}). If (\ref{Q1}) is satisfied, the defining relations $xx$, $xy$, $xz$, $yx$, $zx$, and $h$ of $A$ form a Gr\"obner basis of the ideal $I=I(V,R)$. If $h=yz-azy$ or $h=yz-zy+zz$, then the normal words of degree $n\geq 2$ are $z^ky^{n-k}$ for $0\leq k\leq n$. Since there are $n+1$ of them, we have $H_A=H_{10}$. If $h=yy$, then the normal words of degree $n\geq 2$ are all monomials in $y$ and $z$, which do not contain $yy$ as a submonomial. It is easy to see that the number $a_n$ of such monomials satisfies the recurrent relation $a_{n+2}=a_{n+1}+a_n$, which together with $a_2=3$ and $a_3=5$ implies $H_A=H_{11}$.

Next, assume that (\ref{Q2}) is satisfied. That is, $A$ is given by the relations $xx$, $xy$, $xz$, $yy$, $zy$ and $h$
with $h\in\{yx-zz,yz-zx,yx,yz,zx,zz\}$. If $h=yx-zz$, the defining relations together with $yzz$, $zzx$, $zzy$ and $zzz$ form a Gr\"obner basis of $I$. The only normal word of degree $\geq 3$ is $yzx$, which gives $H_A=H_2$. If $h=yz-zx$, then the defining relations together with $zzx$ form a Gr\"obner basis of $I$. The only normal word of degree $n\geq 3$ is $z^n$, which implies $H_A=H_3$. If $h\in \{yx,yz,zx,zz\}$, $A$ is monomial and therefore the defining relations form a Gr\"obner basis of $I$. If $h=zz$, the only normal word of degree $\geq 3$ is $yzx$ and $H_A=H_2$. If $h=zx$, the only normal words of degree $n\geq 3$ are $z^n$ and $yz^{n-1}$, which gives $H_A=H_7$. If $h=yz$, the only normal words of degree $n\geq 3$ are $z^{n-1}x$ and $z^n$, yielding $H_A=H_7$. If $h=yx$, the only normal words of degree $n\geq 3$ are $yz^{n-2}x$ $yz^{n-1}$, $z^{n-1}x$ and $z^n$. Hence $H_A=H_9$.

Finally, assume that (\ref{Q3}) is satisfied. That is, $A$ is given by the relations $xx$, $xy$, $xz$, $yy$, $zz$ and $h$
with $h\in\{yx+zx,yz+zy,yx+zx-yz-zy\}$. If $h=yx+zx-yz-zy$, the defining relations together with $yzx$, $yzy$ and $zyz$ form a Gr\"obner basis of $I$. There are no normal words of degree $\geq 3$ and therefore $H_A=H_1$. If $h=yz+zy$, the defining relations form a Gr\"obner basis of $I$. The only normal word of degree $\geq 3$ is $zyx$ and $H_A=H_2$. Finally, if $h=yx+zx$, the defining relations together with $yzx$ form a Gr\"obner basis of $I$. For $n\geq 3$ there are exactly 2 normal words of degree $n$ being the monomials in $y$ and $z$ in which $y$ and $z$ alternate: $yzyz\dots$ and $zyzy\dots$ Hence $H_A=H_7$. Since we have exhausted all the options, the proof is complete.
\end{proof}

\begin{proof}[Proof of Lemma~$\ref{III}$] The fact that $R$ is $6$-dimensional is straightforward. Indeed, $R$ is the image of the $6$-dimensional space of the symmetric elements of $V\otimes V$ under the invertible linear map $I\otimes T$. Now, replacing the ground field by a field extension does not change the Hilbert series of an algebra given by generators and relations. Hence, without loss of generality we can assume that $\K$ is algebraically closed. This allows us to pick a basis $x_1,x_2,x_3$ in $V$ with respect to which the matrix of $T$ has the Jordan normal form.

If $T$ has $3$ Jordan blocks, $T$ has the diagonal matrix with respect to the basis $x_1$, $x_2$, $x_3$ with the non-zero numbers (eigenvalues) $\lambda_1$, $\lambda_2$ and $\lambda_3$ on the diagonal. One easily sees that in this case $R$ is spanned by $x_j^2$ with $1\leq j\leq 3$ and $\lambda_k x_jx_k+\lambda_j x_kx_j$ with $1\leq j<k\leq 3$. If $T$ has $2$ Jordan blocks, we can assume that the size two block is in the left upper corner and corresponds to the eigenvalue $\lambda$, while the size one block is in the right lower corner and corresponds to the eigenvalue $\mu$. In this case $x_1^2$, $\lambda x_2^2+x_2x_1$, $x_3^2$, $x_1x_2+x_2x_1$, $\mu x_1x_3+\lambda x_3x_1$ and $\mu x_2x_3+x_3x_1+\lambda x_3x_2$ forms a linear basis in $R$. Finally, if $T$ has just one Jordan block corresponding to the eigenvalue $\lambda$, a linear basis in $R$ is formed by $x_1^2$,  $\lambda x_2^2+x_2x_1$, $\lambda x_3x_1+\lambda x_3x_2+x_3^2$, $x_1x_2+x_2x_1$, $x_1x_3+x_3x_1-x_2x_1$ and $x_2x_3+x_3x_2+x_3x_1$. In any case, this linear basis in $R$ is also a Gr\"obner basis in $I(V,R)$ with the  only normal word of degree $\geq 3$ being $x_3x_2x_1$. This gives $H_A=H_2$.
\end{proof}

This completes the proof of Proposition~\ref{adse}. Note that if $\K$ is algebraically closed and  $A$ is a quadratic algebra satisfying $H_A=H_j$ for $j\in \{8,9,10,11\}$, Lemma~\ref{split} can be applied and $A$ falls into one of the cases considered in the proofs of Lemmas~\ref{I} and~\ref{II}. Scanning the proofs, one sees that whenever $H_A=H_j$ for $j\in \{8,9,10,11\}$, $A$ is actually PBW and therefore Koszul. Since the Hilbert series or Koszulity do not notice an extension of the ground field, we can drop the condition that $\K$ is algebraically closed. This observation automatically implies the following  Koszulity result.

\begin{proposition}\label{autopbw} If $A$ is a quadratic algebra satisfying $H_A=H_j$ for $j\in \{8,9,10,11\}$, then $A$ is Koszul. Moreover, $A$ is PBW provided $\K$ is algebraically closed.
\end{proposition}

\section{Proof of Lemma~\ref{split}}

We start by reformulating Lemma~\ref{split}. First, if we take a pairing on $V\otimes V$ as in the definition of a dual algebra, then in terms of $S=R^\perp$, Lemma~\ref{split} reads in the following way.

\begin{lemma}\label{split2} Let $V$ be a $3$-dimensional vector space over an infinite field $\K$ and $S$ be a $3$-dimensional subspace of $V\otimes V$. Then at least one of the following statements is true$:$
\begin{itemize}\itemsep=-3pt
\item[{\rm (P$1'$)}] there is a $2$-dimensional subspace $M\subset V$ such that $(V{\otimes}M){\oplus}S{=}V{\otimes}V$ or $(M{\otimes}V){\oplus}S{=}V{\otimes}V;$
\item[{\rm (P$2'$)}] there is a $2$-dimensional subspace $M\subset V$ such that $V\otimes M\supset S$ or $M\otimes V\supset S;$
\item[{\rm (P$3'$)}] there is an invertible linear map $T:V\to V$ such that $S=\spann\{x\otimes Ty-y\otimes Tx:x,y\in V\}$.
\end{itemize}
\end{lemma}

For two vector spaces $V_1$ and $V_2$ over $\K$, $L(V_1,V_2)$ stands for the vector space of all linear maps from $V_1$ to $V_2$. Using the natural isomorphism between $V\otimes V$ and $L(V^*,V)$ together with the fact that a two-dimensional subspace of $V$ is exactly the kernel of a non-zero linear functional, we can reformulate Lemma~\ref{split2} in the following way.

\begin{lemma}\label{split3} Let $V$ be a $3$-dimensional vector space over an infinite field $\K$ and $S$ be a $3$-dimensional subspace of $L(V^*,V)$. Then at least one of the following statements is true$:$
\begin{itemize}\itemsep=-3pt
\item[{\rm (P$1''$)}] there is $f\in V^*$ such that the map $A\mapsto Af$ or the map $A\mapsto A^*f$ from $S$ to $V$ is injective$;$
\item[{\rm (P$2''$)}] there is a non-zero $f\in V^*$ such that $Af=0$ for all $A\in S$ or $A^*f=0$ for all $A\in S;$
\item[{\rm (P$3''$)}] there is an invertible $T\in L(V,V)$ such that $g(TAf)=-f(TAg)$ for all $f,g\in V^*$ and $A\in S$.
\end{itemize}
\end{lemma}

In other words, we have to show that (P$3''$) holds if both (P$1''$) and (P$2''$) fail. Hence, Lemma~\ref{split3} and therefore Lemma~\ref{split} will follow if we prove the following result.

\begin{lemma}\label{split1} Let $V_1$ and $V_2$ be a $3$-dimensional vector spaces over an infinite field $\K$ and $S$ be a $3$-dimensional subspace of $L(V_1,V_2)$. Assume also that
\begin{itemize}\itemsep-2pt
\item[\rm(L1)]$\bigcap\limits_{A\in S}\ker A=\bigcap\limits_{A\in S}\ker A^*=\{0\};$
\item[\rm(L2)]$\{Au:A\in S\}\neq V_2$ for each $u\in V_1$ and $\{A^*f:A\in S\}\neq V_1^*$ for each $f\in V_2^*$.
\end{itemize}
Then there exist linear bases in $V_1$ and $V_2$ such that $S$ in the corresponding matrix form is exactly the space of $3\times 3$ antisymmetric matrices.
\end{lemma}

\begin{proof} First, we shall show that
\begin{equation}\label{ee1}
\text{for every non-zero $x\in V_1$, the space $Sx=\{Ax:A\in S\}$ is two-dimensional.}
\end{equation}
By (L1), $Sx\neq\{0\}$ for each $x\in V_1\setminus\{0\}$. By (L2), $Sx\neq V_2$ for each $x\in V_1$. Thus $Sx$ for $x\in V_1\setminus \{0\}$ is either one-dimensional or two-dimensional. Assume that (\ref{ee1}) fails. Then there is $x_1\in V_1$ such that $Sx_1$ is one-dimensional. Then we can pick a basis $A_1,A_2,A_3$ in $S$ such that $A_1x_1=y_1\neq 0$ and $A_2x_1=A_3x_1=0$. By (L1), the linear span of the images of all $A\in S$ is $V_2$. Hence we can pick $x_2,x_3\in V_1$ such that $x_1,x_2,x_3$ is a basis in $V_1$, while $y_1,y_2,y_3$ is a basis in $V_2$, where $y_j=A_jx_j$. With respect to the bases $x_1,x_2,x_3$ and $y_1,y_2,y_3$, the matrices of $A_1$, $A_2$ and $A_3$ have the shape
$$
\left(\begin{array}{ccc}1&*&*\\0&*&*\\0&*&*\end{array}\right),\ \
\left(\begin{array}{ccc}0&0&*\\0&1&*\\0&0&*\end{array}\right)\ \ \text{and}\ \
\left(\begin{array}{ccc}0&*&0\\0&*&0\\0&*&1\end{array}\right),\ \ \text{respectively.}
$$
Keeping the basis in $V_2$ as well as $x_1$ and replacing $x_2$ and $x_3$ by $x_2+\alpha x_1$ and $x_3+\beta x_1$ respectively with appropriate $\alpha,\beta\in \K$, we can kill the second and the third entries in the first row of the first matrix. With respect to the new basis, the matrices of $A_1$, $A_2$ and $A_3$ are
$$
\left(\begin{array}{ccc}1&0&0\\0&a_1&a_2\\0&a_3&a_4\end{array}\right),\ \
\left(\begin{array}{ccc}0&0&a_5\\0&1&a_6\\0&0&a_7\end{array}\right)\ \ \text{and}\ \
\left(\begin{array}{ccc}0&a_8&0\\0&a_9&0\\0&a_{10}&1\end{array}\right)\ \ \text{with $a_j\in\K$.}
$$
By (L2), for every $u=(x,y,z)\in\K^3$, $A_1u,A_2u,A_3u$ are linearly dependent and $A^T_1u,A^T_2u,A^T_3u$ are linearly dependent as well (here $A_j$ stand for the matrices of the linear maps $A_j$), where $T$ denotes the transpose matrix. Computing these vectors, we see that these conditions read
$$
{\rm det}\left(\begin{array}{ccc}x&a_1y+a_2z&a_3y+a_4z\\a_5z&y+a_6z&a_7z\\a_8y&a_9y&a_{10}y+z\end{array}\right)=
{\rm det}\left(\begin{array}{ccc}x&a_1y+a_3z&a_2y+a_4z\\0&y&a_5x+a_6y+a_7z\\0&a_8x+a_9y+a_{10}z&z\end{array}\right)=0
$$
for all $x,y,z\in\K$. Since $\K$ is infinite, the above two determinants must be zero as polynomials in $x,y,z$. The first determinant has the shape $x(a_{10}y^2+a_6z^2+(1+a_6a_{10}-a_7a_9)yz)+g$ with $g\in\K[y,z]$. Hence $a_{10}=a_6=0$ and $a_7a_9=1$. Taking into account that $a_{10}=a_6=0$, we see that the $xyz$-coefficient of the second determinant is $a_7a_9$. Hence $a_7a_9=0$, which contradicts $a_7a_9=1$. This contradiction proves (\ref{ee1}).

Now we pick a non-zero $u\in V_1$. By (L2), there is a non-zero $A_1\in S$ such that $A_1u=0$. Since $A_1\neq 0$,  there is $x\in V_1$ such that $A_1x\neq 0$. Since a Zarisski generic $x$ will do, we can assure the extra condition $Su\neq Sx$ (otherwise (L1) is violated). Obviously, $u$ and $x$ are linearly independent. By (\ref{ee1}), we can find $A_2\in S$ such that $A_1x$ and $A_2x$ are linearly independent. Again  suppose $A_2$ is Zarisski generic  and therefore we can achieve the extra condition that $A_1x$, $A_2x$ and $A_2u$ are linearly independent (otherwise $Su=Sx$). Now $y_1=-A_2u$, $y_2=A_1x$ and $y_3=A_2x$ form a basis in $V_2$. By (L2), there is a non-zero $A_3\in S$ such that $A_3x=0$. Clearly, $A_j$ are linearly independent (=they form a basis in $S$). Pick a basis $x_1,x_2,x_3$ in $V_1$ such that $x_1=x$ and $x_3=u$.
With respect to the bases $x_1,x_2,x_3$ and $y_1,y_2,y_3$, the matrices of $A_1$, $A_2$ and $A_3$ have the form
$$
\left(\begin{array}{ccc}0&*&0\\1&*&0\\0&*&0\end{array}\right),\ \
\left(\begin{array}{ccc}0&*&-1\\0&*&0\\1&*&0\end{array}\right)\ \ \text{and}\ \
\left(\begin{array}{ccc}0&*&*\\0&*&*\\0&*&*\end{array}\right),\ \ \text{respectively.}
$$
Keeping the basis in $V_2$ as well as $x_1$ and $x_3$ and replacing $x_2$ by $x_2+\alpha x_1+\beta x_3$ with appropriate $\alpha,\beta\in \K$, we can kill the middle entry in the first matrix and the second entry of the first row of the second matrix. With respect to the new basis, the matrices of $A_1$, $A_2$ and $A_3$ have shape:
$$
\left(\begin{array}{ccc}0&a_1&0\\1&0&0\\0&a_2&0\end{array}\right),\ \
\left(\begin{array}{ccc}0&0&-1\\0&a_4&0\\1&a_5&0\end{array}\right)\ \ \text{and}\ \
\left(\begin{array}{ccc}0&a_6&a_9\\0&a_7&a_{10}\\0&a_8&a_{11}\end{array}\right)\ \ \text{with $a_j\in\K$.}
$$
By (L2), for every vector $u=(x,y,z)\in\K^3$, $A_1u,A_2u,A_3u$ are linearly dependent and $A^T_1u,A^T_2u,A^T_3u$ are linearly dependent. Computing these vectors, we see that these conditions mean:
$$
{\rm det}\left(\begin{array}{ccc}a_1y&x&a_2y\\-z&a_4y&x+a_5y\\a_6y+a_9z&a_7y+a_{10}z&a_8y+a_{11}z\end{array}\!\!\right)=
{\rm det} \left(\begin{array}{ccc}y&a_1x+a_2z&0\\-z&a_4y+a_5z&x\\0&a_6x+a_7y+a_{8}z&a_9x+a_{10}y+a_{11}z\end{array}\!\!\right)=0
$$
for all $x,y,z\in\K$. Since $\K$ is infinite, these two determinants are zero as polynomials in $x,y,z$. The terms containing $x^2$ in the first determinant amount to $x^2(a_6y+a_9z)$, which implies $a_6=a_9=0$. The $yz^2$-coefficient of the same polynomial is $-a_2a_{10}$. Hence $a_2a_{10}=0$. First, we show that $a_2=0$. Indeed, assume the contrary. Then $a_2\neq 0$ and therefore $a_{10}=0$. Now $z^3$-coefficient in the second determinant is $-a_2a_{11}$. Hence $a_{11}=0$. Taking into account that $a_6=a_9=a_{10}=a_{11}=0$, we see that in the first determinant the terms containing $z$ amount to $z(a_8xy-a_2a_7y^2)$. It follows that $a_7=a_8=0$ and therefore $A_3=0$, which is a contradiction. Hence $a_2=0$. Recall that we already know that $a_6=a_9=0$. Next, we show that $a_1\neq 0$. Indeed, assume the contrary: $a_1=0$. Then the first determinant simplifies to $zx(a_8y+a_{11}z)$. Hence $a_8=a_{11}=0$. Now the second determinant simplifies to $a_4a_{10}y^3+a_5a_{10}y^2z-a_7xy^2$. Hence $a_7=0$ and $a_4a_{10}=a_5a_{10}=0$. If $a_{10}=0$, we have $A_3=0$, which is a contradiction. If $a_{10}\neq 0$, we have $a_4=a_5=0$. In this case the second column of each $A_j$ is zero. Thus the second basic vector in $V_1$ is in the common kernel of all elements of $S$, which contradicts (L1). These contradictions prove that $a_1\neq 0$. By normalizing the second basic vector in $V_1$ appropriately, we can assume that $a_1=-1$. Taking this into account together with $a_2=a_6=a_9=0$, we see that the $xy^2$ and $xz^2$ coefficients in the first determinant are $-a_7$ and $a_{11}$ respectively. Hence $a_7=a_{11}=0$. Now the determinants in the above display simplify to
$$
-a_4a_8y^3+(a_8+a_{10})xyz+a_5a_{10}y^2z\ \ \text{and}\ \  a_4a_{10}y^3+a_5a_{10}y^2z-(a_8+a_{10})xyz.
$$
Since they vanish, $a_{10}=-a_8$. If $a_8=0$, then $a_{10}=0$ and $A_3=0$, which is a contradiction. Thus $a_8\neq 0$.
Now vanishing of the polynomials in the above display implies $a_4=a_5=0$. Hence, the matrices $A_1$, $A_2$ and $A_3$ acquire the shape
$$
\left(\begin{array}{ccc}0&-1&0\\1&0&0\\0&0&0\end{array}\right),\ \
\left(\begin{array}{ccc}0&0&-1\\0&0&0\\1&0&0\end{array}\right)\ \ \text{and}\ \
a_8\left(\begin{array}{ccc}0&0&0\\0&0&-1\\0&1&0\end{array}\right)\ \ \text{with $a_8\neq 0$}
$$
and $S$ becomes the space of all antisymmetric matrices.
\end{proof}

Since Lemma~\ref{split1} is equivalent to Lemma~\ref{split}, the proof of Lemma~\ref{split} is now complete.

\section{Specific algebras satisfying $H_A=H_j$}

For each $j\in\{1,\dots,11\}$, we provide a quadratic algebra $A_j$ (generated by degree $1$ elements $x$, $y$ and $z$) satisfying $H_{A_j}=H_j$. In each case the last equality is an easy exercise since we supply the finite Gr\"obner basis in the ideal of relations and describe the normal words. These 11 examples are presented in Table~1.

\begin{table}[ht]
\caption{Algebras $A_j$ for $1\leq j\leq 11$\huhu}
\centering
\begin{tabular}{|l|l|l|l|l|}
\hline
\smash{\lower6pt\hbox{$j$}}&\smash{\lower6pt\hbox{defining relations of $A_j$}}&other elements of&normal words&Hilbert\\
&&the Gr\"obner basis&of degrees $\geq 3$&series\\
\hline
 1&$xx{-}zx$, $xy{-}zz$, $xz$, $yx$, $yy{-}zy$, $yz$&$zzx$, $zzy$, $zzz$&none&$H_1$\\
\hline
2&$xx$,  $yx$, $yy$, $zx$, $zy$, $zz$&none&$xyz$&$H_2$\\
\hline
3&$xx{-}zx$, $xy$, $xz$, $yx$, $yy{-}zy$, $yz$&$zzx$, $zzy$&$z^n$ for $n\geq 3$&$H_3$\\
\hline
4&$xx$, $xy$, $xz{-}zz$, $yx$, $yy$, $yz{-}zz$&$zzz$&$zzx$, $zzy$&$H_4$\\
\hline
5&$xx{-}yx$, $xy$, $yy$, $yz$, $zx$, $zz$&$zyx$&$xzy$, $yxz$, $yxzy$&$H_5$\\
\hline
6&$xz{-}yz$, $xy$, $yx$, $yy$, $zx$, $zz$&$zyz$&$yzy$, $x^n$ for $n\geq 3$&$H_6$\\
\hline
7&$xx$, $xy$, $xz$, $yy$, $zx$, $zy$&none&$z^{n}$, $yz^{n-1}$ for $n\geq 3$&$H_7$\\
\hline
8&$xy$, $xz$, $yx$, $yz$, $zx$, $zy$&none&$x^{n}$, $y^{n}$, $z^{n}$ for $n\geq 3$&$H_8$\\
\hline
9&$xx$, $xz$, $yx$, $zx$, $zy$, $zz$&none&$xy^{n-1}$, $xy^{n-2}z$, $y^n$,&$H_9$\\
&&&$y^{n-1}z$ for $n\geq 3$&\\
\hline
10&$xx$, $xy$, $xz$, $yx$, $yz$, $zx$&none&$z^my^{n-m}$ for $n\geq 3$,&$H_{10}$\\
&&&$0\leq m\leq n$&\\
\hline
11&$xx$, $xy$, $xz$, $yx$, $yy$, $zx$&none&all monomials in $y$, $z$&$H_{11}$\\
&&&without $yy$ as a subword&\\
\hline
\end{tabular}
\label{table1}
\end{table}

Counting normal words is trivial for all $A_j$ except the last one, where the normal words of degree $n\geq 3$ are exactly monomials in $z$ and $y$, which do not contain $yy$ as a submonomial. In this case, as it was already observed in the proof of Lemma~\ref{II}, the numbers $a_n$ of such monomials of degree $n$ are consecutive Fibonacci  numbers with $a_3=5$, yielding  $H_{A_{11}}=H_{11}$.

\begin{proof}[Proof of Theorem~$\ref{main001}$] By Proposition~\ref{adse}, $H_A\in \{H_1,\dots,H_{11}\}$ for every quadratic $\K$-algebra $A$ satisfying $\dim A_1=\dim A_2=3$. The examples in Table~1 show that each $H_j$ with $1\leq j\leq11$ is indeed the Hilbert series of a quadratic $\K$-algebra.
\end{proof}

It remains to deal with Koszulity. We need the following elementary observation.

\begin{lemma}\label{now} Assume that $A$ is a degree graded algebra on generators $x_1,\dots,x_n$, that the monomials in $x_j$ carry a well-ordering compatible with the multiplication and that $\Lambda$ is the set of the leading monomials of all members of the corresponding Gr\"obner basis of the ideal of relations of $A$. Let also $1\leq j,k\leq n$ be such that $x_j\neq 0$, $x_k\neq 0$ and $x_jx_k=0$ in $A$. Finally, assume that $\Lambda$ contains no monomial ending with $x_sx_k$ for $s\neq j$. Then for $u\in A$, $ux_k=0\iff u=vx_j$ for some $v\in A$.
\end{lemma}

\begin{proof} Since $x_jx_k=0$ in $A$, $ux_k=vx_jx_k=0$ if $u=vx_j$ for some $v\in A$.

Assume now that $u\in A$ and $ux_k=0$. It remains to show that $u=vx_j$ for some $v\in A$. Let ${\cal N}$ be the set of all normal words for $A$. That is, ${\cal N}$ is the set of monomials containing no member of $\Lambda$ as a submonomial. Since ${\cal N}$ is a linear basis in $A$, we can write $u$ as a linear combination of elements of ${\cal N}$. We also separate those words in this combination ending with $x_j$ from the rest of them:
$u=\sum\limits_\alpha c_\alpha w_\alpha x_j+\sum\limits_\beta d_\beta v_\beta$, where both sums are finite (an empty sum is supposed to be zero), $w_\alpha x_j$, $v_\beta$ are pairwise distinct normal words, none of $v_\beta$ ends with $x_j$ and $c_\alpha,d_\beta\in\K^*$. Then
$$\textstyle
0=ux_k=\sum\limits_\alpha c_\alpha w_\alpha x_jx_k+\sum\limits_\beta d_\beta v_\beta x_k=\sum\limits_\beta d_\beta v_\beta x_k\ \ \text{in $A$},
$$
where the last equality is due to $x_jx_k=0$. Since $v_\beta$ does not end with $x_j$ and $\Lambda$ contains neither $x_j$ nor $x_k$ nor any monomial ending with $x_sx_k$ with $s\neq j$, we easily see that each $v_\beta x_k$ is a normal word. Since the set of normal words is linearly independent in $A$, the above display implies that the sum $\sum\limits_\beta d_\beta v_\beta$ is empty and therefore $u=\sum\limits_\alpha c_\alpha w_\alpha x_j=vx_j$ with $v=\sum\limits_\alpha c_\alpha w_\alpha$.
\end{proof}

\begin{proposition} \label{kokoko} The algebras $A_j$ with $j\in\{2,5,6,7,8,9,10,11\}$ are Koszul.
\end{proposition}

\begin{proof} For $j\in\{2,7,8,9,10,11\}$, $A_j$ is a monomial algebra, hence  it is 
 Koszul. It remains to verify Koszulity of $A_5$ and $A_6$. Consider the algebra $B$ given by the generators $x,y,z$ and the relations $xx+yx$, $xz$, $zy$ and the algebra $C$ given by the generators $x,y,z$ and the relations $xx$, $xz+yz$, $zy$.
A direct computation shows that $B^!=A_5$ and $C^!=A_6$. Hence Koszulity of $A_5$ and $A_6$ is equivalent to Koszulity of $B$ and $C$ respectively. It remains to prove that $B$ and $C$ are Koszul, which is our objective now.

Consider the following sequences of free graded left $B$-modules and $C$-modules:
\begin{align}\label{koco10}
&\qquad0\to B\mathop{\longrightarrow}^{d_4}B^2\mathop{\longrightarrow}^{d_3}B^3\mathop{\longrightarrow}^{d_2}B^3
\mathop{\longrightarrow}^{d_1}B\mathop{\longrightarrow}^{d_0}\K\to 0,
\\
\label{koco1}
&\dots\mathop{\longrightarrow}^{\delta_5}C\mathop{\longrightarrow}^{\delta_5}C\mathop{\longrightarrow}^{\delta_5}C
\mathop{\longrightarrow}^{\delta_4}C^2\mathop{\longrightarrow}^{\delta_3}C^3\mathop{\longrightarrow}^{\delta_2}C^3
\mathop{\longrightarrow}^{\delta_1}C\mathop{\longrightarrow}^{\delta_0}\K\to 0,
\end{align}
where $d_0$ and $\delta_0$ are the augmentation maps,
\begin{align*}
&\text{$d_1(u,v,w)=ux+vy+wz$, $d_2(u,v,w)=(u(x+y),vz,wx)$, $d_3(u,v)=(0,ux,v(x+y))$,}
\\
&\text{$d_4(u)=(u(x+y),0)$, $\delta_1(u,v,w)=ux+vy+wz$, $\delta_2(u,v,w)=(ux,vz,w(x+y))$,}
\\
&\text{$\delta_3(u,v)=(ux,v(x+y),0)$, $\delta_4(u)=(ux,0)$ and $\delta_5(u)=ux$.}
\end{align*}
Using the relations of $B$ and $C$, one easily sees that the composition of any two consecutive arrows in both sequences is indeed zero. By definition of Koszulity, the proof will be complete if we show that these sequences are exact. The exactness of (\ref{koco10}) and (\ref{koco1}) boils down to verifying the following statements:
\begin{align}\label{exac01}
&\text{for $u\in B$, $u(x+y)=0\iff u=0$,}
\\
\label{exac02}
&\text{for $u\in B$, $ux=0\iff u=v(x+y)$ for some $v\in B$,}
\\
\label{exac03}
&\text{for $u\in B$, $uz=0\iff u=vx$ for some $v\in B$,}
\\
\label{exac1}
&\text{for $u\in C$, $ux=0\iff u=vx$ for some $v\in C$,}
\\
\label{exac2}
&\text{for $u\in C$, $u(x+y)=0\iff u=0$,}
\\
\label{exac3}
&\text{for $u\in C$, $uz=0\iff u=v(x+y)$ for some $v\in C$.}
\end{align}
Indeed, the exactness of (\ref{koco10}) at the leftmost $B$ is equivalent to (\ref{exac01}), its exactness at $B^2$ is equivalent to (\ref{exac01}) and (\ref{exac02}) and its exactness at the leftmost $B^3$ is equivalent to (\ref{exac01}), (\ref{exac02}) and (\ref{exac03}). The exactness of (\ref{koco1}) at each $C$ which is to the left of $C^2$ is equivalent to (\ref{exac1}), its exactness at $C^2$ is equivalent to (\ref{exac1}) and (\ref{exac2}), while its exactness at the leftmost $C^3$ is equivalent to (\ref{exac1}), (\ref{exac2}) and (\ref{exac3}). Checking the exactness of both complexes at three terms on the right is a straightforward exercise. Alternatively, one can notice that (\ref{koco10}) and (\ref{koco1}) are the Koszul complexes of $B$ and $C$ respectively, and that the Koszul complex happens to be exact at three right terms for every quadratic algebra, see \cite{popo} (exactness of the Koszul complex at the two right terms holds for every graded algebra generated in degree $1$, while the exactness at the term third from the right holds for all quadratic algebras).

Thus it remains to prove (\ref{exac01} - \ref{exac3}). Observe that the defining relations $xx+yx$, $xz$, $zy$ of $B$ together with $xy^kx+y^{k+1}x, \, k\geq 1$ form the Gr\"obner basis of the ideal of relations of $B$, while the defining relations $xx$, $xz+yz$, $zy$ of $C$ together with $xyz$ form the Gr\"obner basis of the ideal of relations of $C$. Now a direct application of Lemma~\ref{now} justifies (\ref{exac03}) and (\ref{exac1}). In order to prove the rest, we perform the following linear substitution. Keeping $x$ and $z$ as they were, we set the new $y$ to be $x+y$ in the old variables. This substitution provides an isomorphism of $B$ and the algebra $A$ given by the generators $x$, $y$, $z$ and the relations $xz$, $yx$, $zx-zy$. These relations together with $zyz$ form the Gr\"obner basis of the ideal of relations of $A$. The same substitution provides an isomorphism of $C$ and the algebra $D$ given by the generators $x$, $y$, $z$ and the relations $xx$, $yz$, $zx-zy$. These relations together with $zyx$ form the Gr\"obner basis of the ideal of relations of $D$.

Now we can rewrite (\ref{exac01}), (\ref{exac02}), (\ref{exac2}) and (\ref{exac3}) in terms of multiplication in $A$ and $D$. Namely, they are equivalent to
\begin{align}\label{exac010}
&\text{for $u\in A$, $uy=0\iff u=0$},
\\
\label{exac020}
&\text{for $u\in A$, $ux=0\iff u=vy$ for some $v\in A$,}
\\
\label{exac21}
&\text{for $u\in D$, $uy=0\iff u=0$,}
\\
\label{exac31}
&\text{for $u\in D$, $uz=0\iff u=vy$ for some $v\in D$,}
\end{align}
respectively. Again, Lemma~\ref{now} justifies (\ref{exac020}) and (\ref{exac31}). Next, one easily sees that the sets of normal words for both $A$ and $D$ are closed under the multiplication by $y$ on the right. This implies (\ref{exac010}) and (\ref{exac21}). Hence (\ref{exac01} - \ref{exac3}) hold and therefore $B$ and $C$ are Koszul.
\end{proof}

\begin{proposition}\label{noko} Let $A$ be a quadratic algebra such that $H_A\in\{H_1,H_3,H_4\}$. Then $A$ is non-Koszul.
\end{proposition}

\begin{proof} Assume the contrary. Then $A$ is Koszul and by (\ref{stm2}), $H_{A^!}(t)=\frac{1}{H_A(-t)}$. In particular, all coefficients of the series $\frac{1}{H_A(-t)}$ must be non-negative. On the other hand,
$$\textstyle
\frac1{H_1(-t)}=1+3t+6t^3+9t^3+9t^4-27t^6+{\cdots}\ \ \text{and}\ \
\frac1{H_3(-t)}=1+3t+6t^3+10t^3+14t^4+16t^5+12t^6-4t^7+{\cdots}
$$
Hence $H_1$ and $H_3$ can not be Hilbert series of a Koszul algebra.

It remains to consider the case $H_A=H_4$. Since replacing the ground field by a field extension does not effect the Hilbert series or Koszulity, we can without loss of generality assume that $\K$ is algebraically closed. By Lemmas~\ref{split}, \ref{II} and \ref{III}, $A=A(V,R)$ with $R$ satisfying condition (P1) of Lemma~\ref{split}. Thus, by passing to the algebra with the opposite multiplication, if necessary, we can assume that $R\oplus(L\otimes V)=V\otimes V$ for a 1-dimensional subspace $L$ of $V$. Now choose a basis $x$, $y$, $z$ in $V$ such that $x$ spans $L$. Then $R^\perp \oplus(M\otimes V)=V\otimes V$, where $M=\spann\{y,z\}$. It follows that we can choose a basis $f$, $g$, $h$ in $R^\perp$ such that the leading monomials of $f$, $g$ and $h$ are $xx$, $xy$ and $xz$ respectively. Since $H_A=H_4=1+3t+3t^2+2t^3$, a direct computation shows that
$$\textstyle
H_{A^!}=\frac1{H_A(-t)}=1+3t+6t^2+11t^3+21t^4+42t^5+85t^6+{\cdots}
$$
(we need few first coefficients). Since $\dim A_3^!=11$, there should be exactly one degree $3$ element $q$ of the Gr\"obner basis of the ideal of relations of $A^!$. The leading monomial $\overline{q}$ of $q$ can have either the shape $u_1u_2x$ or the shape $u_1u_2u_3$, where $u_j\in\{y,z\}$. First, assume that $\overline{q}=u_1u_2u_3$. Then $A^!_4$ is spanned by $v_1v_2v_3x$ with $v_j\in\{y,z\}$ and $v_1v_2v_3\neq u_1u_2u_3$ and by $v_1v_2v_3v_4$ with $v_j\in\{y,z\}$ and $v_1v_2v_3\neq u_1u_2u_3$, $v_2v_3v_4\neq u_1u_2u_3$. The number of these monomials is $20$ if $u_1=u_2=u_3$ and is $19$ otherwise. Thus $\dim A^!_4\leq 20$. Since by the above display $\dim A^!_4=21$, we have arrived to a contradiction, which proves that $\overline{q}$ can not be of the shape $u_1u_2u_3$.

Hence $\overline{q}=u_1u_2x$ with $u_1,u_2\in\{y,z\}$. In this case, were the relations $f$, $g$, $h$ together with $q$ is the Gr\"obner basis, the dimension of $A^!_{4}$ would have been $22$. Since $\dim A^!_4=21$, there is exactly one degree $4$ element $p$ of the Gr\"obner basis of the ideal of relations of $A^!$. The leading monomial $\overline{p}$ of $p$ can have either the shape $w_1w_2w_3x$ or the shape $w_1w_2w_3w_4$, where $w_j\in\{y,z\}$. Again, first, assume that $\overline{p}=w_1w_2w_3w_4$. Then $A^!_5$ is spanned by $v_1v_2v_3v_4x$ with $v_j\in\{y,z\}$ and $v_3v_4\neq u_1u_2$ and by $v_1v_2v_3v_4v_5$ with $v_j\in\{y,z\}$ and $v_1v_2v_3v_4\neq w_1w_2w_3w_4$, $v_2v_3v_4v_5\neq w_1w_2w_3w_4$. It easily follows that $\dim A^!_5\leq 41$. Since by the above display $\dim A^!_5=42$, we have arrived to a contradiction, which proves that $\overline{p}$ can not be of the shape $w_1w_2w_3w_4$.

Hence $\overline{p}=w_1w_2w_3x$ with $w_j\in\{y,z\}$ and $w_2w_3\neq u_1u_2$. In this case, $A^!_6$ is spanned by 64 elements $v_1v_2v_3v_4v_5v_6$ with $v_j\in\{y,z\}$ and by 20 elements $v_1v_2v_3v_4v_5x$ with $v_j\in\{y,z\}$, $v_4v_5\neq u_1u_2$, $v_3v_4v_5\neq w_1w_2w_3$. Hence $\dim A^!_6\leq 84$. Since by the above display $\dim A^!_6=85$, we have arrived to a contradiction. Thus $H_4$ is not the Hilbert series of a Koszul algebra.
\end{proof}

\begin{proof}[Proof of Theorem~$\ref{main002}$] By Proposition~\ref{kokoko}, for $j\in \{2,5,6,7,8,9,10,11\}$, there is a Koszul algebra $A$ satisfying $H_A=H_j$. By Proposition~\ref{noko}, every quadratic algebra $A$ satisfying $H_A=H_j$ with $j\in\{1,3,4\}$ is non-Koszul.
\end{proof}

\subsection{Some remarks}

{\bf 1.} \ The condition of $\K$ being infinite in Lemma~\ref{split} can be relaxed to $\K$ having sufficiently many elements. More precisely, examining closely the idea behind the proof, one gets that Lemma~\ref{split} holds if the condition of $\K$ being infinite is relaxed to $|\K|\geq 4$. On the other hand, the following example shows that the conclusion of Lemma~\ref{split} fails if $|\K|=2$.

\begin{example}\label{Z2} Let $x$, $y$, $z$ be a basis of a $3$-dimensional vector space $V$ over the $2$-element field $\K=\Z_2$. Let also $R\subset V\otimes V$ be the linear span of $x\otimes x$, $y\otimes y$, $z\otimes z$, $y\otimes z+z\otimes y$, $x\otimes y+z\otimes x+z\otimes y$ and $x\otimes z+y\otimes x+z\otimes y$. Then $R$ is a $6$-dimensional subspace of $V\otimes V$ for which each of the conditions {\rm (P1--P3)} of Lemma~$\ref{split}$ fails.
\end{example}

We leave the verification to the reader. It can be done by brute force since $V^+=V\setminus\{0\}$ has just $7$ elements. For example, to show that (P1) fails, one has to find for every $u\in V^+$, $\,\,v,w\in V^+$ such that $u\otimes v,w\otimes u\in R$. Note that extending $\K$ to a $4$-element field forces $R$ from the above example to satisfy (P1). We do not know whether the conclusion of Lemma~\ref{split} holds if $|\K|=3$.

\medskip

\noindent {\bf 2.} \ By Proposition~\ref{autopbw}, $A$ is automatically Koszul if $H_A=H_j$ for $j\in \{8,9,10,11\}$.
 Note that if $A$ is a quadratic algebra satisfying $H_A=H_j$ with $j\in\{2,5,6,7\}$, this does not necessarily mean that $A$ is Koszul. We construct the following examples to illustrate this. In these examples we assume $|\K|>2$ and $\alpha\in\K$ is an arbitrary element different from $0$ or $1$. Table~2 is completed by computing the Gr\"obner bases of ideals of relations of algebras $B_j$.

\begin{table}[ht]
\caption{Algebras $B_j$ for $j\in\{2,5,6,7\}$\huhu}
\centering
\scalebox{0.95}{
\begin{tabular}{|l|l|l|l|l|}
\hline
\smash{\lower6pt\hbox{$j$}}&\smash{\lower6pt\hbox{relations of $B_j$}}&other elements of&\smash{\lower6pt\hbox{$H_{B_j}$}}&\smash{\lower6pt\hbox{relations of $B_j^!$}}\\
&&the Gr\"obner basis&&\\
\hline
2&$xx{+}yz$,\,$xz$,\,$yx$,\,$yy{+}zx$,\,$zy$,\,$zz$&$yzx{-}zxy$,\,$xyz{-}zxy$&$H_2$&$xx{-}yz$,\,$xy$,\,$yy{-}zx$\\
\hline
\smash{\lower6pt\hbox{5}}&$xx{-}zx$,\,$xy{-}zx$,\,$yx{-}zx$,\,$yy{-}zx$,&\smash{\lower6pt\hbox{$zzx$,\,$zzzz$}}&\smash{\lower6pt\hbox{$H_5$}}&$zy$,\,$xz{+}yz{+}(1{-}\alpha)zz$,\\
&$xz{+}\alpha zx{-}\frac{1}{1{-}\alpha}zz$,\,$yz{+}\alpha zx{-}\frac{1}{1{-}\alpha}zz$&&& $xx{+}xy{+}yx{+}yy{+}zx{+}\alpha(1{-}\alpha)zz$\\
\hline
6&$xx{-}\alpha zx$,\,$xy{-}zy$,\,$yx$,\,$xz{-}\alpha zx$,\,$yz$,\,$yy$&$zzx$,\,$zzzy$&$H_6$&$xx+xz+\frac1{\alpha}zx$,\,$xy+zy$,\,$zz$\\
\hline
7&$xx{-}zx$,\,$xy$,\,$yx$,\,$xz{-}\alpha zx$,\,$yz$,\,$yy$&$zzx$&$H_7$&$xx+\alpha xz+zx$,\,$zy$,\,$zz$\\
\hline
\end{tabular}
}
\label{table2}
\end{table}

Computing the Gr\"obner bases of ideals of relations of algebras $B^!_j$ up to degree $4$, we easily obtain the data presented in Table~3.

\begin{table}[ht]
\caption{The series $(H_{B_j}(-t))^{-1}$ and $H_{B_j^!}(t)$ up to degree $4$\huhu}
\centering
\begin{tabular}{|l|l|l|}
\hline
$j$&$(H_{B_j}(-t))^{-1}$ up to $t^4$&$H_{B_j^!}(t)$ up to $t^4$\\
\hline
2&$1+3t+6t^2+10t^3+15t^4+{\cdots}$&$1+3t+6t^2+10t^3+17t^4+{\cdots}$\\
\hline
5&$1+3t+6t^2+11t^3+20t^4+{\cdots}$&$1+3t+6t^2+11t^3+21t^4+{\cdots}$\\
\hline
6&$1+3t+6t^2+11t^3+20t^4+{\cdots}$&$1+3t+6t^2+11t^3+21t^4+{\cdots}$\\
\hline
7&$1+3t+6t^2+11t^3+19t^4+{\cdots}$&$1+3t+6t^2+11t^3+20t^4+{\cdots}$\\
\hline
\end{tabular}
\label{table3}
\end{table}

Table~3 ensures that each $B_j$ fails to satisfy $H_{B_j}(-t)H_{B_j^!}(t)=1$ and therefore each $B_j$ is non-Koszul. Thus the following statement holds true.

\begin{proposition}\label{SSS} Assuming $|\K|>2$, for each $j\in\{2,5,6,7\}$,
there is a non-Koszul quadratic algebra $B$ satisfying $H_{B}=H_j$.
\end{proposition}

\noindent {\bf 3.} \ By the duality formula (\ref{stm2}), Theorem~\ref{main002} implies the list of all Hilbert series of Koszul algebras $A$ satisfying $\dim A_1=3$ and $\dim A_2=6$. They are the series $1/H(-t)$ for $H$ from the list specified in Theorem~\ref{main002}. Thus we have the following corollary.

\begin{corollary}\label{36}For Koszul algebras $A$ satisfying $\dim A_1=3$ and $\dim A_2=6$, the complete list of Hilbert series consists of $\frac{1}{H_j(-t)}$ with $j\in\{2,5,6,7,8,9,10,11\}$.
\end{corollary}

\medskip

\noindent {\bf 4.} \ A quadratic algebra $A$ satisfying $H_A(t)H_{A^!}(-t)=1$ is called {\it numerically Koszul}. There are examples of  numerically Koszul quadratic algebras which are not Koszul, see \cite{popo}. While proving Proposition~\ref{noko}, we have actually shown that $H_j$ with $j\in\{1,3,4\}$ can not be the Hilbert series of a numerically Koszul quadratic algebra. We do not know an answer to the following question.

\begin{question}\label{q1} Let $A$ be a numerically Koszul quadratic algebra satisfying $\dim A_1=\dim A_2=3$. Is it true that $A$ is Koszul$?$
\end{question}

We are especially interested in the following particular case.

\begin{question}\label{q2}\footnote{In the meantime we have acquired an affirmative solution of this question} Let $A$ be a quadratic algebra satisfying $H_A(t)=(1-t)^{-3}$ and $H_{A^!}(t)=(1+t)^3$. Is it true that $A$ is Koszul$?$
\end{question}

Note that $(1-t)^{-3}$ is the Hilbert series of $\K[x,y,z]$. The following example shows that for a quadratic algebra $A$, the equality $H_A=(1-t)^{-3}$ alone does not guarantee numeric Koszulity.

\begin{example}\label{nonuko} Let $A$ be the quadratic algebra given by the generators $x$, $y$, $z$ and the relations $xx$, $xz+yy+zx$, $xy+yx+zz$. Then $H_A=(1-t)^{-3}$, while $H_{A^!}=H_3$. In particular, $A$ is not numerically Koszul and therefore is non-Koszul.
\end{example}

\begin{proof} A direct computation shows that the defining relations of $A$ together with $yyz-zyy$ and $yzz-zyy$ form a Gr\"obner basis for the ideal of relations of $A$. Now one easily sees that the normal words for $A$ are $z^k(yz)^ly^mx^\epsilon$ with $k,l,m\in\Z_+$ and $\epsilon\in\{0,1\}$ and that the number of normal words of degree $n$ is $\frac{(n+1)(n+2)}2$. Hence $H_A(t)=(1-t)^{-3}$. The dual $A^!$ is given by the relations $yz$, $zy$, $yy-zx$, $xz-zx$, $xy-zz$ and $yx-zz$, which together with $zxx$, $zzx$ and $zzz$ form a Gr\"obner basis for the ideal of relations of $A^!$. The only normal word of degree $n\geq 3$ is $x^n$, which gives $H_{A^!}(t)=H_3$.
\end{proof}

\medskip

\noindent {\bf 5.} \ The following question remains open.

\begin{question}\label{q3} Which series feature as the Hilbert series of quadratic algebras satisfying $\dim A_1=3$ and $\dim A_2=4?$ Which of these occur for Koszul $A?$
\end{question}

The answer to the above question would complete the list of Hilbert series of Koszul algebras $A$ satisfying $\dim A_1=3$.
In \cite{popo}, it is mentioned that it is unknown whether there is a Koszul algebra $A$ satisfying $\dim A_1=3$, $\dim A_2=4$ and $\dim A_3=3$.
It is important to answer also because if
 such an algebra  exists, it would provide a counterexample to the conjecture on rationality of the Hilbert series of Koszul modules over Koszul algebras.
 However if such an algebra does not exist, nothing can be derived about rationality.

\medskip

{\bf Acknowledgements:} \

We are grateful to IHES and MPIM for hospitality, support, and excellent research atmosphere.
We would like to thank anonymous referees for careful reading and useful comments.
This work is funded by the ERC grant 320974, and partially supported by the project PUT9038.


\small\rm

\vspace{15mm}

\scshape

\noindent   Natalia Iyudu

\noindent School of Mathematics

\noindent  The University of Edinburgh

\noindent James Clerk Maxwell Building

\noindent The King's Buildings

\noindent Peter Guthrie Tait Road

\noindent Edinburgh

\noindent Scotland EH9 3FD

\noindent E-mail address: \qquad {\tt niyudu@staffmail.ed.ac.uk}\ \ \

\vskip 11mm

\vskip1truecm

\scshape

\noindent  Stanislav Shkarin

\noindent Queens's University Belfast

\noindent Department of Pure Mathematics

\noindent University road, Belfast, BT7 1NN, UK

\noindent E-mail addresses: \qquad {\tt s.shkarin@qub.ac.uk}\ \ \


\end{document}